\documentclass[15pt]{amsart}
\usepackage[all,cmtip]{xy}
\usepackage[normalem]{ulem}
\usepackage{amsthm}
\usepackage{amssymb}
\usepackage{amsmath}
\usepackage{xypic}
\usepackage{graphicx}
\usepackage{amscd}
\usepackage{setspace}
\usepackage[colorlinks=false, linkcolor=red]{hyperref}
\usepackage{xcolor}
\usepackage{mathtools}

\newcommand{\C}{\mathbb{C}}
\newcommand{\R}{\mathbb{R}}
\newcommand{\E}{\mathcal{E}}

\newcommand{\F}{\mathcal{F}}

\newcommand{\OO}{\mathcal{O}}

\newtheorem{theorem}{Theorem}[section]
\newtheorem{thm}[theorem]{Theorem}
\newtheorem{cor}[theorem]{Corollary}
\newtheorem{lemm}[theorem]{Lemma}
\newtheorem{prop}[theorem]{Proposition}
\newtheorem{claim}[theorem]{Claim}
\newtheorem{warning}[theorem]{Warning}
\DeclareMathOperator {\Ho}{H}
\DeclareMathOperator {\Def}{Def}
\DeclareMathOperator {\FEt}{F\acute{E}t}
\DeclareMathOperator {\Ext}{Ext}
\DeclareMathOperator {\Hom}{Hom}
\DeclareMathOperator {\Spec}{Spec}
\DeclareMathOperator {\Aut}{Aut}
\DeclareMathOperator {\Id}{Id}
\DeclareMathOperator {\At}{At}
\DeclareMathOperator {\Gr}{Gr}
\DeclareMathOperator {\rank}{rank}

\newtheorem{defi}[theorem]{Definition}

\newtheorem{remark}[theorem]{Remark}

\numberwithin{equation}{theorem}

\title []{Unobstructedness of deformations of Calabi--Yau varieties with a line bundle}

\author{Shizhang Li}
\address{Department of Mathematics, Columbia University, MC 4406, 2990 Broadway,
New York, NY, 10027, U.S.}
\email{shanbei@math.columbia.edu}
\author {Xuanyu Pan}
\address{Institute of Mathematics, AMSS, Chinese Academy of Sciences, 55 ZhongGuanCun East Road, Beijing, 100190, China}
\email{pan@amss.ac.cn}

\subjclass{Primary 13D10; Secondary 14C30}
\keywords{Deformation Theory, Hodge theory, Polarized Calabi--Yau Varieties}

\begin{document}
\maketitle
\dedicatory

\begin{abstract}
We generalize the Tian--Todorov Theorem in 
the case of Calabi--Yau varieties equipped with a line bundle.
\end{abstract}
\tableofcontents

\section {Introduction}
Let $X$ be a smooth and proper variety having torsion canonical bundle, and let $L$ be a line bundle on it. The goal of this paper is to prove that the deformations of the pair $(X,L)$ are unobstructed,
thus generalizing a famous result of Tian--Todorov.

 \begin{thm}\cite{TG,TO} \label{Thmtiantod}
 If $X$ is a compact K\"{a}hler manifold with trivial canonical bundle, 
 then the Kuranishi space of deformations of complex structures on $X$ is smooth. In a parallel analogy, 
 if $X$ is projective with an ample class $w$, 
 then the Kuranishi space of deformations of $X$ with $w$ is also smooth.
 \end{thm}

Namely we prove

\begin{thm}\label{mainthm}
Let $X$ be a smooth and proper variety with torsion canonical bundle and let $L$ be a line bundle on $X$. The deformations of the pair $(X,L)$ are unobstructed.

\end{thm}


One might attempt to solve the problem of the deformational unobstructedness by proving the vanishing of certain obstruction groups.
However, 
for some Calabi--Yau manifold $X$, the obstruction group $\Ho^2(X,T_X)$ can be huge. 
To overcome this difficulty, there are two known methods. The first one is the so called $T_1$-lifting criterion, which was introduced by Kawamata and Ran \cite{KAW, RAN} to provide an alternative algebraic proof of the
Tian--Todorov Theorem for Calabi--Yau varieties.
Moreover, it also applies to
 deformation problems with a hull \cite{AR, V}.
The second method relies on the $dg BV$ algebras, which was introduced by Kontsevich~\cite{Kon}  to establish some generalized Tian--Todorov theorems. 
Nevertheless, it seems that both the aforementioned two methods fail to establish Theorem~\ref{mainthm}.

Now, we present our geometrical approach as follows. Firstly, we use the Beauville--Bogomolov decomposition theorem to reduce Theorem~\ref{mainthm} to the case of abelian varieties with a line bundle. Next, by establishing the infinitesimal variational Hodge conjecture for line bundles \cite[Conjecture 1.4]{defcycle}, the problem reduces to the smoothness of certain Hodge loci in period domains, which can be solved by explicit calculations in linear algebra. Lastly, we remark that Grothendieck, Mumford and Oort studied the unobstructedness problem of the deformations of an abelian variety with an ample line bundle in mixed characteristic \cite[Theorem 2.4.1]{Oort}, while it is still open when the line bundle is not ample. Our theorem for the abelian varieties gives a positive answer to this problem over complex numbers.

\textbf{Acknowledgments.} 
Both authors are very grateful to their advisor Prof.~A.~J.~de Jong.
The second named author also thanks his friend Zhiyu Tian for a lot of inspiration and encouragement.
He also thanks Prof.~M.~Abouzaid, Prof.~M.~Kerr, Prof.~J.~Starr and Prof.~M.~Manetti for discussions and pointing out references.


\section{Sketch of the proof}
Now, let us present the structure of this paper. 

Let $X$ be a smooth and proper variety with torsion canonical bundle, and let $L$ be a line bundle on $X$. We will prove some basic facts concerning the deformation theory of the pair $(X,L)$ in Section \ref{secAO}.

Soon after, some general facts about the compatibility of the obstruction elements with finite \'etale coverings and products are proved in Section \ref{s3} and Section \ref{secobsprod}. Note that, by the Beauville--Bogomolov decomposition theorem, there is a finite \'etale cover $\tilde{X}\rightarrow X$ with $\tilde{X}=A\times Y \times Z$ where $A$ is an abelian variety, $Y$ is a Calabi--Yau manifold with vanishing Hodge numbers and $Z$ is an irreducible holomorphic symplectic manifold. Using the results Proposition \ref{thm1} and Proposition \ref{unobsprod} in these sections, we reduce to prove the theorem for $(A,L_1)$, $(Y,L_2)$ and $(Z,L_3)$ separately. 
The latter two cases are dealt by showing the vanishing of some obstruction elements, see Proposition \ref{unobshk}.
The main difficult part is to show the theorem for the abelian variety $A$ with a line bundle $L_1$.

Let $\Def_A$ (resp.~$\Def_{(A,L)}$) be the deformation functor of $A$ (resp.~the pair $(A,L)$). We study the forgetful functor $\Def_{(A,L)}\rightarrow \Def_A$ which forgets the line bundles, and show the functor and its image are ``formally smooth". One main ingredient of the proof is to show the infinitesimal variational Hodge conjecture for line bundle on a smooth proper formal scheme over rings of power series. 

For the sake of completeness, we recall the infinitesimal variational Hodge conjecture \cite[Conjecture 1.4]{defcycle} here. Let $B$ be a smooth projective scheme over an scheme $S$. Suppose that $S$ is $\Spec(\mathbb{C}[[t]])$. Denote by $B_1$ the scheme $B\times_S \Spec(\mathbb{C}[[t]]/(t))$. There is a Chern character ring homomorphism from the K-group to the de Rham cohomology
\[ ch : K_0(B_1)\rightarrow H^*_{dR}(B_1/\mathbb{C}).\] We denote by $F^r H_{dR}^i\subseteq H_{dR}^i$ the Hodge filtration on the de Rham cohomology. The Gauss-Manin connection $\nabla$ on $H^*_{dR}$ gives a canonical isomorphism
\[ \Phi :H^{*}_{dR}(B/S)^{\nabla=0} \xrightarrow{\cong} H^*_{dR}(B_1/\mathbb{C}).\]

\noindent\textsc{Conjecture.} With the notation above, for an element $\xi\in K_0(B_1)_{\mathbb{Q}}$, if we have \[\Phi^{-1} \circ ch(\xi_1)\in \sum_i H^{2i}_{dR}(B/S)^{\nabla=0}\cap F^i H^{2i}_{dR}(B/S),\] 
then there is an element  $\xi\in K_0(B)_{\mathbb{Q}}$ such that \[ ch(\xi|_{B_1})=ch(\xi_1)\in H^*_{dR}(B_1).\]\\

In Section \ref{secinfhodgeconj}, we will show this conjecture for $\xi_1$ a line bundle on a smooth proper formal scheme over any rings of power series (not just $\Spec(\mathbb{C}[[t]])$), see Theorem \ref{thmlifting} and Remark \ref{rmkdefcycle}. And we conclude that the image of the forget functor $\Def_{(A,L)}\rightarrow \Def_A$ is the ``germ'' of the Hodge locus in the local deformation space of $A$ in Section \ref{secav}. We also show that both the Hodge locus and the forget functor are ``smooth" by some calculations of period domains, see Theorem \ref{linear algebra smooth} and Proposition \ref{deformation surjection}. In summary, we prove the deformations of an abelian variety along with a line bundle are unobstructed, see Theorem \ref{abelian variety unobstructed}. Finally, we will prove our main Theorem~\ref{mainthm} in Section \ref{secmainthm}. 


\section{Atiyah extensions and obstructions} \label{secAO}

Let $X$ be a non-singular variety over the complex numbers. We have a morphism (see~\cite[Chapter III, Exercise 7.4 (c)]{HA})
\[d(log): \OO^*_{X}\rightarrow \Omega_{X}^1\]
by the rule $u\mapsto du/u$. It induces a group homorphism\[c_1:\Ho^1(X,\OO_X^*)\rightarrow \Ho^1(X,\Omega^1_X)=\Ext_{\OO_X}^1(T_X,\OO_X).\]
For a line bundle $L$ on $X$, we associate with $c_1(L)$ an extension class
\begin{equation} \label{atiyah}
0\rightarrow \OO_X\rightarrow \mathcal{E}_L\rightarrow T_X\rightarrow 0.
\end{equation}
This class is the {\it Atiyah extension} of $L$.

Let $U=\{U_{a}\}$ be an affine open covering of $X$ such that $L$ is represented by a system of transition functions $\{f_{ab}\}$ with $f_{ab}\in \Gamma(U_{ab},\OO_X^*)$. Then the Atiyah extension class of $L$ is represented by the $1$-cocycle
\[\bigl(\frac{df_{ab}}{f_{ab}}\bigr)\in Z^1(U,\Omega^1_X).\]
Locally, the sheaf $\mathcal{E}_{L}|_{U_a}$ is isomorphic to $\OO_{U_a}\oplus T_{X}|_{U_a}$. On $U_{ab}$, we identify
a section $(g_a,d_a)$ of $\OO_{U_a}\oplus T_{X}|_{U_a}$ with a section $(g_b,d_b)$ of $\OO_{U_b}\oplus T|_{X}|_{U_b}$ if and only if
\begin{center}
$d_a=d_b$ \qquad and \qquad $g_b-g_a=\frac{d_a f_{ab}}{f_{ab}}$.
\end{center}

We formulate our deformation problem in the following:
\begin{defi}
An infinitesimal deformation of the pair $(X,L)$ over a $\mathbb{C}$-local artinian ring $A$ consists of a pair $(X_A,L_A)$ such that \[\xymatrix{X\ar[r]\ar@{}[dr]|-{\Box} \ar[d] & X_A\ar[d]\\
\Spec(\mathbb{C})\ar[r] & \Spec(A)}
\]
is a cartesian diagram in which $X_A$ is flat over $\Spec(A)$ and $L_A$ is a line bundle on $X_A$ with $L_A|_X=L$.
\end{defi}
There is an obvious equivalence relation among the deformations of the pair $(X,L)$ {given by isomorphisms}. Hence we can define a functor as follows.
\begin{defi} \label{defunobs}
A functor of infinitesimal deformations of $(X,L)$ is the functor
\[\Def_{(X,L)}: \{\mathbb{C}\textit{-local Artinian rings}\} \rightarrow Sets\]
which associates with a $\mathbb{C}$-local Artinian ring $A$ the set
\begin{center}
$\Def_{(X,L)}(A)=$ \{deformations of $(X,L)$ over $A$\}/ $\sim$
\end{center}
We say that $(X,L)$ is unobstructed if for any surjection $A'\rightarrow A$ of $\mathbb{C}$-local Artinian rings, the induced
restriction map $$\Def_{(X,L)}(A')\rightarrow \Def_{(X,L)}(A)$$ is also a surjection.
\end{defi}

We summarize some useful results from~\cite{ES} in the following proposition.
\begin{prop}
The short exact sequence~\ref{atiyah} induces a long exact sequence
\[\xymatrix{\ldots\ar[r] & \Ho^1(X,T_X)\ar[r]^{\cdot c_1(L)} & \Ho^2(X,\OO_X)\ar[r] & \Ho^2(X,\mathcal{E}_L) \ar[r]^{\textrm{forget}}& \Ho^2(X,T_X)\ar[r] &\ldots}.\]
Moreover, for a small extension $0\rightarrow \mathbb{C}\rightarrow A'\rightarrow A\rightarrow 0$
and a deformation $(X_A,L_A)$ of the pair $(X,L)$ over $A$, there exists an obstruction element \[Obs(X_A,L_A)\in \Ho^2(X,\mathcal{E}_L).\] We also have that
\[{\rm forget}[Obs(X_A,L_A)]= Obs(X_A).\]
In particular, the pair $(X,L)$ is unobstructed if and only if the obstruction element $Obs(X_A,L_A)$ is zero for all the small extensions of $A$.
\qed
\end{prop}

This proposition immediately yields
\begin{lemm}\label{unobs}
Let $X$ be  a smooth variety with a line bundle $L$. Suppose that $X$ is unobstructed. If the map
\[\cup c_1(L):\Ho^1(X,T_X) \rightarrow \Ho^2(X,\OO_X)\]
is surjective, then $(X,L)$ is unobstructed.
\qed
\end{lemm}

\section{Unobstructedness of finite \'etale coverings}\label{s3}
In this section, we show that the obstruction element $Obs(X_A,L_A)$ for the pair $(X_A,L_A)$ is preserved when taking a finite \'etale cover
in Proposition~\ref{thm1}. We could use cotangent complexes\cite{I} to give more general results in this section, see Appendix~\ref{app}.

\begin{prop}\label{thm0}
Let $X$ and $Y$ be smooth proper varieties over the complex numbers $\C$. Suppose that we have a map $f:Y\rightarrow X$ which is finite $\acute{e}$tale. Then
\begin{enumerate}
  \item The map $f^*:\Ho^i(X,T_X)\rightarrow \Ho^i(Y,T_Y)$ is injective.
  \item If $X_A$ is a deformation of $X$ over a local $\C$-Artinian ring $A$, then there is a deformation $Y_A$ of $Y$ over $A$ filling into the following cartesian diagram

      \[\xymatrix{
      Y\ar[r]\ar@{}[dr]|-{\Box} \ar[d]_{f} & Y_A\ar[d]^{f_A} \\
      X\ar[r] & X_A.
      }\]
\end{enumerate}

\end{prop}

\begin{proof}
For the first assertion, we have a map
\[T_X\rightarrow f_*T_Y=T_X\otimes_{\OO_X} f_*\OO_Y\]
by the projection formula.
The tangent bundle $T_X$ is a summand of $T_X\otimes_{\OO_X} f_*\OO_Y$. In fact, by the existence of the trace map $\mathrm{Tr}$, we have
\[\xymatrix{\OO_X\ar[r]\ar@/^1pc/[rrr]^{\cdot \deg(f)} & f_*f^*\OO_X\ar@{=}[r] &f_*\OO_Y\ar[r]_{\mathrm{Tr}} &\OO_X}.\]
In other words, the map \[\mathrm{Tr}\circ f^*:\Ho^i(X,T_X)\rightarrow \Ho^i(Y,T_Y)\] is multiplication by the number $\deg(f)$. In particular, the map $f^*$ is injective.

For the second assertion, since $X \subseteq X_A$ is defined by the nilpotent elements, it is clear that \[\FEt(X_A)=\FEt(X),\] where $\FEt(X)$ is the category of finite \'etale coverings, see~\cite{ET} for the details.
\end{proof}

\begin{prop}
\label{thm00}
Let $X$ be a reduced and proper local complete intersection over the complex numbers $\mathbb{C}$. If a morphism \[f:Y\rightarrow X\] is finite $\acute{e}$tale and $X$ is of finite type over $\mathbb{C}$, then we have a natural map
\[\xymatrix{ \Ext^2(\Omega_X^1,\OO_X) \ar[r]^{f^*} & \Ext^2(\Omega_Y^1,\OO_Y)}
\]
which maps $Obs(X_A)$ to $Obs(Y_A)$ for any small extension \[0\rightarrow (t)\rightarrow A'\rightarrow A\rightarrow 0.\]

\end{prop}
\begin{proof}
We recall how the obstruction element $Obs(X_A)$ is constructed. For the small extension of $A$, we have the exact conormal sequence (\cite[Lemma B.10]{ES})
\[0\rightarrow (t)\rightarrow \Omega^1_{A'/\mathbb{C}}\otimes_{A'} A\rightarrow \Omega_{A/\mathbb{C}}^1\rightarrow 0.\]
We pull it back via the structure morphism $p:X_A\rightarrow \Spec(A)$. We obtain an exact sequence
\begin{equation}\label{eql1}
0\rightarrow \OO_{X_A}\rightarrow p^*(\Omega^1_{A'/\mathbb{C}^1}\otimes_{A'} A)\rightarrow p^*\Omega^1_{A/\mathbb{C}}\rightarrow 0.
\end{equation}
By~\cite[Theorem D.28]{ES}, we have an exact sequence
\begin{equation}\label{eql2}
0\rightarrow p^*(\Omega^1_{A/\mathbb{C}})\rightarrow \Omega^1_{X_A/\mathbb{C}} \rightarrow \Omega^1_{X_A/A}\rightarrow 0.
\end{equation}
Composing (\ref{eql1}) and (\ref{eql2}), we obtain a $2$-extension
\[0\rightarrow \OO_X\rightarrow p^*(\Omega^1_{A'/\mathbb{C}}\otimes_{A'} A) \rightarrow \Omega^1_{X_A/\mathbb{C}} \rightarrow \Omega^1_{X_A/A}\rightarrow 0.\]
It defines the obstruction element
\[Obs(X_A)\in \Ext^2_{\OO_{X_A}}(\Omega^1_{X_A/A},\OO_X)=\Ext^2_{\OO_X}(\Omega^1_{X},\OO_X).\]
Since $f$ is \'etale, it is obvious that the natural map $f^*$ maps the $2$-extension
\[0\rightarrow \OO_X \rightarrow E \rightarrow F \rightarrow \Omega^1_{X/\C}\rightarrow 0.\]
to the $2$-extension
\[0\rightarrow f^*\OO_X=\OO_Y \rightarrow f^*E \rightarrow f^*F \rightarrow f^*\Omega^1_{X/\C}=\Omega^1_{Y/\C} \rightarrow 0.\]
So, to prove that $f^*$ preserves the obstruction, it suffices to prove that the exact sequence
\[0\rightarrow (p\circ f_A)^{*}(\Omega^1_A)\rightarrow f_A^*\Omega^1_{X_A} \rightarrow f_A^*\Omega^1_{X_A/A}\rightarrow 0\]
identifies with
\[0\rightarrow (p\circ f_A)^{*}(\Omega^1_A)\rightarrow \Omega^1_{Y_A} \rightarrow \Omega^1_{Y_A/A}\rightarrow 0.\]
This follows from the fact that $f_A$ is \'etale. Indeed, we have exact sequences
\[0\rightarrow f_A^*\Omega^1_{X_A/\mathbb{C}} \rightarrow \Omega^1_{Y_A/\mathbb{C}} \rightarrow \Omega^1_{Y_A/X_A}\rightarrow 0\]
and
\[0\rightarrow f_A^*\Omega^1_{X_A/A} \rightarrow \Omega^1_{Y_A/A} \rightarrow \Omega^1_{Y_A/X_A}\rightarrow 0.\]
Since $\Omega^1_{Y_A/X_A}=0$, we have
\begin{center}
    $f_A^*\Omega^1_{X_A/\mathbb{C}}=\Omega^1_{Y_A/\mathbb{C}}$ \qquad and \qquad $f_A^*\Omega^1_{X_A/A}=\Omega^1_{Y_A/A}$
\end{center}
This concludes the proof of the proposition.
\end{proof}

\begin{cor}
With the same assumptions as in Proposition~\ref{thm0}, the variety $Y$ is unobstructed if and only if $X$ is unobstructed. In particular, if $X$ is a smooth proper variety with torsion canonical bundle, then $X$ is unobstructed.
\end{cor}

\begin{proof}
For any proper smooth variety with torsion canonical bundle, we can find a finite \'etale cover such that its canonical bundle is trivial, the cover is unobstructed by Theorem \ref{Thmtiantod}. Then, the assertion follows from the previous propositions.
\end{proof}

\begin{prop} \label{thm1}
Let $X$ be a smooth proper variety over the complex numbers $\C$ with a line bundle $L$. Suppose that $f:Y\rightarrow X$ is finite \'etale. Then
\begin{enumerate}
  \item The Atiyah extension class is functorial under $f$, i.e., $f^*\mathcal{E}_{L}=\mathcal{E}_{f^*L}$.
  \item The pullback map $f^*:\Ho^i(X,\E_L)\rightarrow \Ho^i(Y,\E_{f^*L})$ is injective.
  \item If we have a small extension $0\rightarrow \C\rightarrow A'\rightarrow A\rightarrow 0$ and a deformation $(X_A,L_A)$ of $(X,L)$, then we have that \[f_A^*(Obs(X_A,L_A))=Obs(Y_A,f_A^*L_A),\]in which the morphism $f_A:Y_A\rightarrow X_A$ is given by Proposition~\ref{thm0} (2).
\end{enumerate}
In particular, the pair $(Y,L)$ is unobstructed if and only if the pair $(X,f^*L)$ is unobstructed, see Definition \ref{defunobs}.
\end{prop}
\begin{proof}
As in Section $2$, the vector bundle $\mathcal{E}_L$ is given by the data trivializations $\OO_{U_a}\oplus T_X|_{U_a}$ and transition functions as follows:
\[\left(
\begin{array}{c}
  g_b \\
  d_b
\end{array}
\right)
=\left(\begin{tabular}{c|c c c}

         1 & & $df_{ab}$ & / $f_{ab}$ \\ \hline
          & &   \\
          & &Id &\\
          & &
       \end{tabular}
\right)
\left(
\begin{array}{c}
  g_a \\
  d_a
\end{array}
\right).
\]
If we pull back the Atiyah extension class of $L$, then we get
\begin{equation}\label{eq1}
0 \rightarrow f^*\OO_X=\OO_Y\rightarrow f^*\mathcal{E}_L\rightarrow f^*T_X=T_Y\rightarrow 0.
\end{equation}
Since the transition function for $f^*L$ is given by
\begin{center}
    $\{f^*(f_{ab})\},f_{ab}\in \Gamma(f^{-1}U_{ab},\OO_Y^*)$
\end{center}
and $f^*(df_{ab})=d(f^*f_{ab})$, we conclude that $\mathcal{E}_{f^*L}$ is the Atiyah extension classes of $f^*L$ by comparing the transition functions of the Atiyah extension classes.

For the second assertion, the proof is similar to that of the second assertion of Proposition~\ref{thm0}.

For the third assertion, we recall how to construct the obstruction element $Obs(X_A,L_A)$.
Suppose that we have a small extension \[0\rightarrow (t)\rightarrow A'\rightarrow A\rightarrow 0\] and an affine covering $\{U_i\}$ of $X$.
We can choose a trivialization \[\theta_i:U_i\times \Spec(A)\xrightarrow{\cong} X_A|_{U_i}\] since $X_A$ is smooth over $A$.
Let $\theta_{ij}$ be $\theta_i^{-1}\theta_j \in \Aut(U_{ij}\times \Spec(A))$. The line bundle $L_A$ is given by transition functions $(F_{ij})$ on $U_{ij}\times A$ such that
\[F_{ij}\theta_{ij}(F_{jk})=F_{ik}.\]
To see whether there is a lifting $(X_{A'},L_{A'})$ of $(X_A,L_A)$ over $\Spec(A')$, we choose a collection $(\theta_{ij}',F_{ij}')$ such that
\begin{enumerate}
  \item $\theta_{ij}'\in \Aut(U_{ij}\times \Spec(A'))$ and $F_{ij}'\in\OO^*_{U_{ij}\times \Spec(A')}$,
  \item $\theta_{ij}'|_{U_{ij}\times A}=\theta_{ij}$ and $F_{ij}'|_{U_{ij}\times A}=F_{ij}.$
\end{enumerate}
Since $\theta_{ij}'\theta_{jk}'(\theta_{ik}')^{-1}|_{U_{ij}\times \Spec(A)}=\Id$, we have
\begin{equation} \label{eqc1}
\theta_{ij}'\theta_{jk}'(\theta_{ik}')^{-1}=\Id+t\cdot d_{ijk} \textbf{ and}
\end{equation}
\begin{equation} \label{eqc2}
F_{ij}'\theta_{ij}'(F_{jk})(F_{ik}')^{-1}=1+t\cdot g_{ijk}
\end{equation}
where $d_{ijk}\in \Gamma(U_{ijk},T_X)$ and $g_{ijk}\in \Gamma(U_{ijk},\OO_X)$.
So the obstruction element $Obs(X_A,L_A)$ can be represented by a $2$-cocycle
\[(g_{ijk},d_{ijk})\in Z^2(U,\mathcal{E}_L).\]
It follows from the infinitesimal lifting property of the \'etale map $Y_A|_{f^{-1}U_i}\rightarrow  X_A|_{U_i}$ that we have a unique trivialization $\widetilde{\theta_i}$ of $Y_A|_{f^{-1}U_i}$ {making the following diagram commute}.
\[\xymatrix{f^{-1}(U_{i})\ar[r]\ar[d] &Y|_{U_i} \ar@{^{(}->}[r] & Y_A|_{f^{-1}U_i}\ar[d]^{\acute{e}t}\\
f^{-1}(U_{i})\times A\ar[r]\ar@{-->}[rru]^{\widetilde{\theta_i}} &U_i\times A\ar[r]^{\theta_i}_{\cong} &X_A|_{U_i}
}\]
Let $\widetilde{\theta_{ij}}$ be $\widetilde{\theta_i}^{-1}\widetilde{\theta_j}$. Similarly, we have $\widetilde{F_{ij}}$, $\widetilde{d}_{ijk}$ and $\widetilde{g}_{ijk}$. Then we have
\[\xymatrix{f^{-1}(U_{ij})\times A\ar[r]^>>>>>{h_A}\ar[d] & U_{ij}\times A\ar[d]\\
f^{-1}(U_{ij})\times A' \ar[r]^>>>>>{h_{A'}} & U_{ij}\times A'
}\]
where $h_A=f\times \Id_{\Spec(A)}$ and $h_{A'}=f\times \Id_{\Spec(A')}$. Hence, we have that
\begin{enumerate}
  \item $h_{A'}^{\sharp}(\theta'_{ij})|_{f^{-1}(U_{ij})\times A} =\widetilde{\theta_{ij}}$
  \item and $h_{A'}^{\sharp}(F'_{ij})|_{f^{-1}(U_{ij})\times A} =h_A^{\sharp}(F_{ij})=\widetilde{F_{ij}}$
\end{enumerate}
where $h_{A'}^{\sharp}(-)$ is the pull-back.
Moreover, we know that $h_A^{\sharp}(\widetilde{F_{ij}})$ defines the line bundle $f_A^*L_A$ on $Y_A$. Combining these facts and applying $h_{A'}^{\sharp}(-)$ to (\ref{eqc1}) and (\ref{eqc2}), we get
\[f^{*}(g_{ijk},d_{ijk})=(\widetilde{g_{ijk}},\widetilde{d_{ijk}}),\]
which concludes the proof.

\end{proof}

\section{Unobstructedness of products} \label{secobsprod}
In this section, we show that the deformations of a product of certain Calabi--Yau manfolds preserve the product structure. It follows the unobstructedness for certain Calabi--Yau varieties with product structures, see Proposition~\ref{unobsprod}.

\begin{lemm}
Let $X$ and $Y$ be smooth and proper varieties. Assume that
\begin{enumerate}
  \item $\Ho^0(Y,T_Y)=\Ho^1(Y,\OO_Y)=0$,
  \item $K_X=\OO_X$ and $K_Y=\OO_Y$ where $K_X$ (resp.~$K_Y$) is the canonical bundle of $X$ (resp.~Y).
\end{enumerate}
Let $Z$ be $X\times Y$. If $Z_A$ is a deformation of $Z$ over a $\C$-local Artinian ring $A$, then we have that
\[Z_A\simeq X_A \times_A Y_A\]
where $X_A$ (resp.~$Y_A$) is a deformation of $X$ (resp.~$Y$) over $A$.
\end{lemm}
\begin{proof}
By the Schlessinger's criterion~\cite{AR}, we know the deformation functors of $X$, $Y$ and $Z$ have hulls as follows
\begin{center}
$h_{R} \twoheadrightarrow \Def_X$, $h_{R'}\twoheadrightarrow \Def_Y$ and $h_{R''}\twoheadrightarrow \Def_Z$,
\end{center}
where $R$, $R'$ and $R''$ are complete local rings over $\mathbb{C}$.
We consider a natural transformation $g$ between functors
\begin{center}
$g:\Def_X\times \Def_Y \rightarrow \Def_{Z}=\Def_{X\times Y}$
\end{center}
which associates with $(X_A,Y_A)\in (\Def_X\times \Def_Y)(A)$
\begin{center}
a deformation $(X_A\times_A Y_A)\in \Def_{X\times Y}(A)$ of $Z$
\end{center}
It induces the following commutative diagram, since $h_{R''} \twoheadrightarrow \Def_{X \times Y}$ is a hull:
\[\xymatrix{h_{R}\times h_{R'} \ar@{->>}[d] \ar[r]^f & h_{R''}\ar@{->>}[d] \\
\Def_X\times \Def_Y \ar[r]^g & \Def_{X\times Y}.}\]
We claim that the natural transformation $f$ induced by $g$ is surjective, hence $g$ itself is surjective.
In fact, we have $h_{R \hat{\otimes} R'}=h_{R}\times h_{R'}$ and $X,Y,Z$ are unobstructed since the canonical bundles are trivial. Note that
\begin{center}
    $R\cong \mathbb{C}[[x_1,\ldots,x_n]]$, $R'\cong \mathbb{C}[[x_1,\ldots,x_m]]$ and $R''\cong \mathbb{C}[[x_1,\ldots,x_s]]$,
\end{center}
i.e., the rings $R$, $R'$ and $R''$ are formal power series rings. Therefore, we only need to check
\[g: \Def_X(\mathbb{C}[\varepsilon])\times \Def_Y(\mathbb{C}[\varepsilon]) \rightarrow \Def_{X\times Y}(\mathbb{C}[\varepsilon])\]
is surjective where $\varepsilon$ is the dual number. By the identification of first order infinitesimal deformation of a variety $W$ with $\Ho^1(W,T_W)$, it suffices to prove that
\[\Ho^1(X,T_X)\oplus \Ho^1(Y,T_Y)\rightarrow \Ho^1(X\times Y,T_{X\times Y})\]
is surjective (in fact, it is an isomorphism) where this map is induced by the projections $\pi_1$ and $\pi_2$.
\[\xymatrix{X &X\times Y\ar[l]_{\pi_1}\ar[r]^{\pi_2}&Y}
\]
Since $T_{X\times Y}=\pi_1^* T_X\oplus \pi_2^* T_Y$, we have that
\[\Ho^1(X\times Y, T_{X\times Y})=\Ho^1(X\times Y, \pi_1^* T_X)\oplus \Ho^1(X\times Y,\pi_2^* T_Y).\]
So it is reduced to show that
\begin{center}
    $\Ho^1(X\times Y ,\pi_1^* T_X)=\Ho^1(X,T_X)$ and $\Ho^1(X\times Y, \pi_2^* T_Y)=\Ho^1(Y,T_Y)$.
\end{center}
By the Leray spectral sequence for $\pi_1$, we have
\[0\rightarrow \Ho^1(X,\pi_{1*}\pi_1^*T_X)\rightarrow \Ho^1(X\times Y, \pi_1^*T_X) \rightarrow \Ho^0(X, R^1\pi_{1*}\pi_1^*T_X).\]
By the projection formula and the base change theorem, we have
\begin{enumerate}
    \item $\pi_{1*}\pi_1^*T_X=T_X\otimes \pi_{1*}\OO_{X\times Y}=T_X$, in particular, $\Ho^1(X,\pi_{1*}\pi_1^*T_X)=\Ho^1(X,T_X)$,
    \item $R^1\pi_{1*}\pi_1^*T_X=T_X \otimes R^1\pi_{1*}\OO_{X\times Y}$,
    \item $R^1\pi_{1*}\OO_{X\times Y}=0$ because of the hypothesis $\Ho^1(Y,\OO_Y)=0$.
\end{enumerate}
In summary, we have $\Ho^1(X,T_X)=\Ho^1(X\times Y,\pi^*_1T_X)$. Similarly, for $Y$, we have
\[0\rightarrow \Ho^1(Y,T_Y)\rightarrow \Ho^1(X\times Y, \pi_2^*(T_Y)) \rightarrow \Ho^0(Y,T_Y\otimes R^1\pi_{2*}\OO_{X\times Y}).\]
Since $R^1\pi_{2*}\OO_{X\times Y}$ is a trivial bundle of rank $N$, we have
\[\Ho^0(Y,T_Y\otimes R^1\pi_{2*}\OO_{X\times Y})=\Ho^0(Y,T_Y)^{\oplus N}=0\]by the hypothesis. We get $\Ho^1(Y,T_Y)=\Ho^1(X\times Y,\pi^*_2 T_Y)$.
\end{proof}

\begin{prop} \label{unobsprod}
Suppose that $X$ and $Y$ satisfy the same conditions as in the lemma above. Let $L$ be a line bundle on $Z=X\times Y$ and $x$ (resp.~$y$) be a closed point of $X$ (resp.~$Y$). Assume that the canonical bundle $K_Z$ is trivial.
If $g=\{x\}\times \Id_Y$ and $f=\Id_X\times \{y\}$ are the natural inclusions via the points $x$ and $y$\[\xymatrix{ &X\times Y \\
X\ar[ur]^{f} & & Y\ar[ul]_g }
\]
and the pairs
\begin{center}
 $(X,f^*L)$ and $(Y,g^*L)$
\end{center}
are unobstructed, then the pair $(Z,L)$ is unobstructed (see Definition \ref{defunobs}).
\end{prop}
\begin{proof}
Let $(Z_A,L_A)$ be a deformation of the pair $(Z,L)$ over a $\mathbb{C}$-local Artinian Ring $A$. Assume that we have a small extension\[0\rightarrow \mathbb{C}\rightarrow A'\rightarrow A\rightarrow 0.\]By the previous lemma, we have a decomposition $Z_A=X_A\times_A Y_A.$ By the smoothness of $Y_A\rightarrow \Spec(A)$, we have a lifting as follows.
\[\xymatrix{ \{y\}\ar@{^{(}->}[d] \ar[r] & Y_A\ar[d]\\
\Spec(A) \ar@{-->}[ru]^{\sigma} \ar[r] & \Spec(A)
}
\]
It induces an embedding
\[i=(\Id_{X_A},\sigma \circ \pi) :X_A \hookrightarrow X_A\times_A Y_A\]
where $\pi$ is the structure morphism of $X_A\rightarrow \Spec(A)$. Similarly, we have an inclusion
\[j :Y_A \hookrightarrow X_A\times_A Y_A.\]
for $Y_A$ via the point $\{x\}.$
Since $(X,L|_{X\times \{y\}})$ and $(Y,L|_{\{x\}\times Y})$ are unobstructed, there exist deformations
\begin{center}
$(X_{A'},L')|_{X_A}=(X_A,L_A|_{X_A})$\qquad and\qquad $(Y_{A'},L'')|_{Y_A}=(Y_A,L_A|_{Y_A})$.
\end{center}
Let $Z_{A'}$ be $X_{A'}\times_{A'} Y_{A'}$. We claim that there is a line bundle $L_{A'}$ on $Z_{A'}$ such that the pair $(Z_{A'},L_{A'})$ is a deformation of $(Z_A,L_A)$ over $A'$. Therefore, the pair $(Z,L)$ is unobstructed.
Note that we have a commutative diagram 
\[\xymatrix{ 0\ar[r] &\OO_Z\ar[d]\ar[r] &\OO_{Z_{A'}}^*\ar[r]\ar[d] & \OO_{Z_{A}}^*\ar[r]\ar[d] &0 \\
 0\ar[r] &\OO_X\ar[r] &\OO_{X_{A'}}^*\ar[r] & \OO_{X_{A}}^*\ar[r] &0 }
\]
with short exact rows. 
It induces the commutative diagram
\[\xymatrix{ \ldots \ar[r]  & \Ho^1(Z_{A'},\OO^*_{Z_{A'}})\ar[r]\ar[d] & \Ho^1(Z_A, \OO_{Z_{A}}^*)\ar[r]\ar[d] & \Ho^2(Z,\OO_Z)\ar[d]^{f^*}\ar[r] &\ldots \\
 \ldots \ar[r] & \Ho^1(X_{A'},\OO_{X_{A'}}^*) \ar[r] & \Ho^1(X_A,\OO_{X_{A}}^*)\ar[r] & \Ho^2(X,\OO_X)\ar[r] & \ldots }
\]
of long exact sequences.
We chase the digram of the second square above to see that
\[\xymatrix{ (Z_A,L_A)\ar@{|->}[d] \ar@{|->}[r] & Obs_L(Z_A,L_A) \ar@{|->}[d]^{f^*}\\
(X_A,L_A|_{X_A})\ar@{|->}[r] & Obs_L(X_A,L_A|_{X_A})=0,}
\]
where $Obs_L(\_)$ is the obstruction. Similarly, we have
\[g^*(Obs_L(Z_A,L_A))=Obs_L(Y_A,L_A|_{Y_A})=0.\]
By the hypothesis $\Ho^1(Y,\OO_Y)=0$ and the K\"unneth formula, we have an isomorphism
\[\xymatrix{(f^*,g^*):\Ho^2(Z,\OO_Z) \ar@{->}[r]^{\simeq} &\Ho^2(\OO_X)\oplus \Ho^2(\OO_Y)}.\]
Therefore, we have $Obs_L(Z_A,L_A)=0$. It implies that there exists a line bundle $L_A'$ on $Z_{A'}$ where $(Z_{A'},L_{A'})$ is a deformation of $(Z_A,L_A)$. 
\end{proof}

\section{Infinitesimal Hodge conjecture for line bundles} \label{secinfhodgeconj}
In this section, we show a general theorem (Theorem~\ref{thmlifting}) confirming the infinitesimal variational Hodge conjecture for line bundles. We start with a theorem due to Deligne.
\begin{theorem}[{\cite[Theorem 3.2]{Bloch} and~\cite{Del}}]
\label{prop0}
Let S be a scheme over $\Spec (\mathbb{Q})$, and let $\pi:X\rightarrow S$ be a proper and smooth morphism. Then
\begin{enumerate} \label{thmlf}
\item The sheaves $R^q \pi_*(\Omega^r_{X/S})$ are locally free of finite type and commute with base change.
\item The spectral sequence 
\[E^{r,q}_1=R^q \pi_*(\Omega^r_{X/S})\Longrightarrow \mathbb{R}^{r+q}\pi_* (\Omega_{X/S}^{\bullet})\]
degenerates at $E_1$.
\item The sheaves $\mathbb{R}^{p}\pi_* (\Omega_{X/S}^{\bullet})$ are locally free of finite type and commute with base change.
\end{enumerate}
\end{theorem}
The following proposition is due to Katz and Bloch.
\begin{prop}[{\cite[Proposition 3.7, Proposition 3.8]{Bloch}}] 
\label{prop2}
\leavevmode
\begin{enumerate}
\item Let $k$ be a field of characteristic $0$, $M$ a finite $k[[t_1,\ldots,t_r]]$-module with integrable connection $\nabla$. Let $M^{\nabla}$ be $Ker(\nabla)$. 
Then $M$ is $M^{\nabla}\otimes_k k[[t_1,\ldots,t_r]]$.
\item Let $A$ be a complete, local, augmented $\mathbb{C}$-algebra (e.g. $A$ artinian), $S=\Spec(A)$, and let $\pi:X\rightarrow S$ be a proper and smooth morphism with  the closed fiber $X_0$. Then $$\mathbb{H}^*(X,\Omega_{X/S}^{\bullet})\cong H^*(X_0,\mathbb{C})\otimes_{\mathbb{C}}A.$$
\end{enumerate}
\end{prop}

In the following, let $A$ be the ring $\mathbb{C}[[x_1,\ldots,x_n]]$ of formal power series with the maximal ideal $m=(x_1,\ldots,x_n)$. Let $X$ be an $m$-adic formal scheme over $\mathrm{Spf}(A)$ Suppose that the structure morphism
\[f:X\rightarrow \mathrm{Spf}(A)\]
is smooth and proper. Denote the sheaf of (continuous) $1$-differentials on the formal scheme $X$ (resp.~$\mathrm{Spf}(A)$) by $\Omega_{X/\mathbb{C}}^{1}$ (resp.~$\Omega_{\mathrm{Spf}(A)/\mathbb{C}}^{1}$). Let $\Omega_{X/A}^1$ be the sheaf of (relative) $1$-differentials for the morphism $X\xrightarrow{f} \mathrm{Spf}(A)$. For the structure sheaf $\OO_X$, we have a morphism $log:1+m\OO_X\rightarrow m\OO_X\text{~with~}$ 
\begin{equation}\label{eqlogmap}
 log(1+x)=\sum\limits_{i=1}^{\infty}(-1)^{n-1}\frac{x^n}{n} \text{~for~} x\in m\OO_X.
 \end{equation}

\begin{lemm} \label{lemmlogiso}
The morphism $log:1+m\OO_X\rightarrow m\OO_X$ (\ref{eqlogmap}) is an isomorphism.
\end{lemm}
\begin{proof}
 It has an inverse $exp:m\OO_X \rightarrow 1+m\OO_X \text{~with~}$  $$exp(x)=\sum\limits_{i=0}^{\infty}\frac{x^n}{n!} \text{~for~} x\in m\OO_X.$$
\end{proof}

On $X$ we have a complex as follows,
\[\OO_X^*\xrightarrow{dlog} \Omega^{1}_{X/A}\xrightarrow{d}  \Omega^{2}_{X/A}\xrightarrow{d}\ldots, \text{~(resp.~}1+m\OO_X\xrightarrow{dlog} \Omega^{1}_{X/A}\xrightarrow{d}  \Omega^{2}_{X/A}\xrightarrow{d}\ldots\text{)}\]
and denote it by $\bigg(\OO^{*}_X\xrightarrow{dlog} \Omega^{\geq 1}_{X/A}\bigg)$ (resp.~$\bigg(1+m\OO_X\xrightarrow{dlog} \Omega^{\geq 1}_{X/A}\bigg)$). The following lemma is clear.
\begin{lemm} \label{lemmlog}
We have a short exact sequence as follows,
\begin{equation}\label{eqshortexseq}
0\rightarrow 1+m\OO_X\rightarrow \OO^{*}_X\rightarrow  \OO^{*}_{X_0} \rightarrow 0
\end{equation}
which induces a short exact sequence of complexes,
\[0\rightarrow \bigg(1+m\OO_X\xrightarrow{dlog} \Omega^{\geq 1}_{X/A}\bigg)\rightarrow \bigg(\OO^{*}_X\xrightarrow{dlog} \Omega^{\geq 1}_{X/A}\bigg)\rightarrow \OO_{X_0}^{*}\rightarrow 0 .\]
The map $log$ (\ref{eqlogmap}) induces an isomorphism of complexes as below,
\[log:\bigg(1+m\OO_X\xrightarrow{dlog} \Omega^{\geq 1}_{X/A}\bigg) \rightarrow \bigg(m\OO_X\xrightarrow{d} \Omega^{\geq 1}_{X/A}\bigg)\]
where the map is identity on $\Omega^{\geq 1}_{X/A}$.
\end{lemm}

We denote the complex $\bigg(1+m\OO_X\xrightarrow{dlog} \Omega^{\geq 1}_{X/A}\bigg)$ (resp.~$\bigg(m\OO_X\xrightarrow{d} \Omega^{\geq 1}_{X/A}\bigg)$) by $\mathrm{Ker}_{\mathrm{dlog}}^{\bullet}$ (resp.~$\mathrm{Ker}^{\bullet}$).

\begin{defi}
We define a map $\Theta: H^1(X_0,\OO_{X_0}^*)\rightarrow \mathbb{H}^2(X,\Omega_{X/A}^{\bullet})$ as the composition of the following maps,
\[H^1(X_0,\OO_{X_0}^*)\xrightarrow{\partial} \mathbb{H}^2(X, \mathrm{Ker}_{\log}^{\bullet})\xrightarrow{log} \mathbb{H}^2(X, \mathrm{Ker}^{\bullet})\rightarrow \mathbb{H}^2(X, \Omega^{\bullet}_{X/A})\]
where the last map is induced by the natural inclusion $\mathrm{Ker}^{\bullet}\hookrightarrow \Omega^{\bullet}_{X/A}$ and $\partial$ is the connecting map induced by the short exact sequence of complexes in Lemma~\ref{lemmlog}. Similarly, we define a map $\theta: H^1(X_0,\OO_{X_0}^*)\rightarrow \mathbb{H}^2(X,\OO_X)$ as the composition of the following maps,
\[H^1(X_0,\OO_{X_0}^*)\xrightarrow{\partial} H^2(X, 1+m\OO_X)\xrightarrow{log} H^2(X, m\OO_X)\rightarrow H^2(X,\OO_X)\]
where the last map is induced by the natural inclusion $m\OO_X\hookrightarrow \OO_X$ and $\partial$ is the connecting map induced by the short exact sequence (\ref{eqshortexseq}).

\end{defi}

\begin{lemm}\label{lemmcomdiag}
With the notations as above, we have a commutative diagram as follows,
\[ \xymatrix{H^1(X_0,\OO_{X_0}^{*})\ar[r]^{\Theta}\ar@{=}[d] &  \mathbb{H}^2(X, \Omega^{\bullet}_{X/A})\ar[d] \\
H^1(X_0,\OO_{X_0}^{*}) \ar[r]^{\theta} & H^2(X,\OO_X) } \]
where the second vertical map is induced by the natural projection $\Omega^{\bullet}_{X/A}\rightarrow \OO_X$.
\end{lemm}
\begin{proof}
This lemma follows from the following commutative diagram:
\[\xymatrix{H^1(X_0,\OO_{X_0}^*)\ar@{=}[d] \ar[r]^{\partial} &\mathbb{H}^2(X, \mathrm{Ker}_{\log}^{\bullet})\ar[d] \ar[r]^{log} &\mathbb{H}^2(X, \mathrm{Ker}^{\bullet})\ar[r]\ar[d] &\mathbb{H}^2(X, \Omega^{\bullet}_{X/A})\ar[d]\\
 H^1(X_0,\OO_{X_0}^*)\ar[r]^{\partial}& H^2(X, 1+m\OO_X)\ar[r]^{log} &H^2(X, m\OO_X)\ar[r] &H^2(X,\OO_X)}
\]
where the natural vertical arrows are induced by the corresponding projections.
\end{proof}

\begin{lemm}\label{lemminjection}
The composition $H^2(X, 1+m\OO_X)\xrightarrow{log} H^2(X, m\OO_X)\rightarrow H^2(X,\OO_X)$ is injective.
\end{lemm}

\begin{proof}
By Theorem~\ref{prop0}, the map $H^*(X,\OO_X)\rightarrow H^*(X_0,\OO_{X_0})$ is surjective. It follows that the map $H^2(X, m\OO_X)\rightarrow H^2(X,\OO_X)$ is injective. Therefore, this lemma follows from Lemma~\ref{lemmlogiso}.
\end{proof}

\begin{lemm}\label{lemmhorizfil}
Let $\nabla$ be the Gauss-Manin connection on $\mathbb{H}^2(X, \Omega^{\bullet}_{X/A})$, and let $L_0 (\in H^1(X_0,\OO_{X_0}^*))$ be a line bundle on $X_0$. Then
\begin{enumerate}
\item  the composition $\nabla\circ \Theta$ is zero,
\item and the reduction of $ \Theta(L_0)$ to the de Rham cohomology $\mathbb{H}^2(X_0, \Omega^{\bullet}_{X_0/\mathbb{C}})$ of $X_0$ is the first Chern class $c_1(L_0)$ of the line bundle $L_0$. 
\end{enumerate}
In other words, $\Theta(L_0)$ is the horizontal lifting of $c_1(L_0)\in \mathbb{H}^2(X_0, \Omega^{\bullet}_{X_0/\mathbb{C}})$.
\end{lemm}
\begin{proof}
Let us recall the Gauss-Manin connection on $\mathbb{H}^2(X, \Omega^{\bullet}_{X/A})$. We have a short exact sequence of locally free $\OO_X$-module as follows
\[0\rightarrow f^*\Omega^1_{\mathrm{Spf}(A)/\mathbb{C}}\rightarrow \Omega^1_{X/\mathbb{C}}\rightarrow \Omega^1_{X/A}\rightarrow 0.\]
It induces a decreasing filtration $L^{\bullet}$ on the complex $\Omega^{\bullet}_{X/\mathbb{C}}$ which gives a short exact sequence 
\begin{equation}\label{eqshortex}
0\rightarrow gr_{L}^1(\Omega^{\bullet}_{X/\mathbb{C}}) \rightarrow \Omega^{\bullet}_{X/\mathbb{C}}/L^2 \rightarrow gr_L^0(\Omega^{\bullet}_{X/\mathbb{C}})\rightarrow 0
\end{equation}
where $gr_{L}^1(\Omega^{\bullet}_{X/\mathbb{C}})$ is $f^*\bigg(\Omega^1_{\mathrm{Spf}(A)/\mathbb{C}}\bigg)\otimes \Omega^{\bullet}_{X/A}[-1]$ and  $gr_{L}^0(\Omega^{\bullet}_{X/\mathbb{C}})$ is $ \Omega^{\bullet}_{X/A}$. The Gauss-Manin connection 
\[\nabla: \mathbb{H}^2(X,\Omega_{X/A}^{\bullet})\rightarrow  \mathbb{H}^2(X,\Omega_{X/A}^{\bullet})\otimes \Omega^1_{\mathrm{Spf}(A)/\mathbb{C}}  \]
is given by the connecting map induced by the short exact sequence (\ref{eqshortex}). Therefore, the composition of $\nabla$ and $\mathbb{H}^2(X,\Omega_{X/\mathbb{C}}^{\bullet}) \rightarrow \mathbb{H}^2(X,\Omega_{X/A}^{\bullet})$ is zero.

\begin{enumerate}
\item
Consider the following two complexes $\mathcal{A}$ and $\mathcal{B}$, and a morphism $log_{\mathcal{AB}}$ between them:
\[log_{\mathcal{AB}}:\mathcal{A}=\bigg(1+m\OO_X\xrightarrow{dlog} \Omega^{\geq 1}_{X/\mathbb{C}}\bigg) \rightarrow \mathcal{B}=\bigg(m\OO_X\xrightarrow{d} \Omega^{\geq 1}_{X/\mathbb{C}}\bigg)\]
where the map is identity on $\Omega^{\geq 1}_{X/\mathbb{C}}$.
We have a commutative diagram of short exact sequences (see Lemma~\ref{lemmlog}) as follows:
\[\xymatrix{ 0\ar[r]& \bigg(1+m\OO_X\xrightarrow{dlog} \Omega^{\geq 1}_{X/\mathbb{C}}\bigg)\ar[r] \ar[d]& \bigg(\OO^{*}_X\xrightarrow{dlog} \Omega^{\geq 1}_{X/\mathbb{C}}\bigg)\ar[r]\ar[d] &\OO_{X_0}^{*}\ar[r]\ar@{=}[d] &0\\
0\ar[r]& \bigg(1+m\OO_X\xrightarrow{dlog} \Omega^{\geq 1}_{X/A}\bigg)\ar[r]& \bigg(\OO^{*}_X\xrightarrow{dlog} \Omega^{\geq 1}_{X/A}\bigg)\ar[r] &\OO_{X_0}^{*}\ar[r] &0.
}\]
It induces the following commutative diagram.
\[\xymatrix{H^1(X_0,\OO_{X_0}^*)\ar@{=}[d] \ar[r]^{\partial} &\mathbb{H}^2(X,\mathcal{A})\ar[d] \ar[r]^{log_{\mathcal{AB}}} &\mathbb{H}^2(X, \mathcal{B})\ar[r]\ar[d] &\mathbb{H}^2(X, \Omega^{\bullet}_{X/\mathbb{C}})\ar[d]\\
 H^1(X_0,\OO_{X_0}^*)\ar[r]^{\partial}& H^2(X,  \mathrm{Ker}_{\log}^{\bullet})\ar[r]^{log} &H^2(X, \mathrm{Ker}^{\bullet})\ar[r] &\mathbb{H}^2(X, \Omega^{\bullet}_{X/A}) }
\]
Note that the composition of $\nabla$ and the map $\mathbb{H}^2(X,\Omega_{X/\mathbb{C}}^{\bullet}) \rightarrow \mathbb{H}^2(X,\Omega_{X/A}^{\bullet})$ is zero by the construction of the Gauss-Manin connection. It follows the first assertion. 

\item
Take a covering $\underline{U}=\{U_i\}$ of $X_0$. Let the $\check{C}$ech cocycle $(f_{ij})$ be the representative of $L_0$ in $H^1(X_0,\OO_{X_0}^*)$, and let $\widetilde{f_{ij}}$ be the lifting of $f_{ij}\in \OO_{X_0}(U_{ij})$ to $\OO_X(U_{ij})$. The composition $log\circ \partial$ of the connecting map $\partial$ and $log$ maps $(f_{ij})$ to the $\check{C}$ech cocycle
\begin{equation}\label{eqcocyle}
\bigg( \big(\frac{d\widetilde{f_{ij}}}{\widetilde{f_{ij}}}\big)_{ij}, \big(log(\widetilde{f_{ij}}\widetilde{f_{jk}}^{-1} \widetilde{f_{ki}})\big)_{ijk}\bigg) \in \Omega^{1}_{X/A}(U_{ij})\oplus m\OO_{X}(U_{ijk}).
\end{equation}
Recall that the first Chern class $c_1(L_0)\in \mathbb{H}^2(X_0,  \Omega^{\geq 1}_{X_0/\mathbb{C}})\subseteq  \mathbb{H}^2(X_0,  \Omega^{\bullet}_{X_0/\mathbb{C}})$ of the line bundle $L_0$ is given by the connecting map of the following exact sequence,
\[0\rightarrow  \Omega^{\geq 1}_{X_0/\mathbb{C}}[-1] \rightarrow\bigg(\OO^{*}_{X_0}\xrightarrow{dlog} \Omega^{\geq 1}_{X_0/\mathbb{C}}\bigg) \rightarrow \OO_{X_0}^*\rightarrow 0.\]
Therefore, the reduction of the cocycle (\ref{eqcocyle}) module $m$  \[ \big(\frac{df_{ij}}{f_{ij}}\big) \in \Omega^1_{X_0/\mathbb{C}}(U_{ij}) \]
gives a $\check{C}ech$ representative of the reduction of $\Theta(L_0)$ to $\mathbb{H}^2(X_0, \Omega^{\bullet}_{X_0/\mathbb{C}})$ which is the cocycle for $c_1(L_0)$ in $\mathbb{H}^2(X_0, \Omega^{\bullet}_{X_0/\mathbb{C}})$. It follows the second assertion.
\end{enumerate}

\end{proof}

\begin{theorem}
\label{thmlifting}
Let $A$ be the ring $\mathbb{C}[[x_1,\ldots,x_n]]$ with the maximal ideal $m=(x_1,\ldots,x_n)$ as the ideal of definition. Let $X$ be an $m$-adic formal scheme over $\mathrm{Spf}(A)$. Suppose that the structure morphism $f:X\rightarrow \mathrm{Spf}(A)$ is proper and formally smooth. If $L_0$ is a line bundle on the special fiber $X_0/\mathbb{C}$ of $f$, then $L_0$ is a restriction of a line bundle on $X$ if and only if the horizontal lifting of the first Chern class $c_1(L_0)$ is in the Hodge filtration $\mathrm{F}^1 \mathbb{H}^2(X,\Omega_{X/A}^{\bullet})$. 
\end{theorem}

\begin{proof}
By Theorem~\ref{prop2} and Lemma~\ref{lemmhorizfil} (1), the element $\Theta(L_0)$ is the unique horizontal lifting of $c_1(L_0)\in \mathbb{H}^2(X_0,\Omega_{X_0}^{\bullet})$. By Lemma~\ref{lemminjection}, we have an exact sequence as follows.
\[\mathrm{Pic}(X)\rightarrow \mathrm{Pic}(X_0)\xrightarrow{\theta} H^2(X,\OO_X)\] 
By Lemma~\ref{lemmcomdiag}, we have a commutative diagram
\[ \xymatrix{ && \mathrm{F}^1 \mathbb{H}^2(X,\Omega_{X/A}^{\bullet})\ar@{^{(}->}[d]\\
& \mathrm{Pic}(X_0)\ar[r]^{\Theta} \ar@{=}[d] & \mathbb{H}^2(X,\Omega_{X/A}^{\bullet})\ar@{->>}[d] \\
\mathrm{Pic}(X)\ar[r] \ar[ru]&\mathrm{Pic}(X_0)\ar[r]^{\theta} &  H^2(X,\OO_X) }\]
where the right vertical sequence is exact by Theorem~\ref{prop0}. Therefore, we conclude the theorem by a simple diagram chase.
\end{proof}

\begin{remark}\label{rmkdefcycle}
The theorem above recovers a corollary of \cite[Theorem 1.2]{defcycle} if $A$ is $\mathbb{C}[[t]]$ and $X$ is an abelian scheme over $A$.
\end{remark}

\section{The smoothness of Hodge loci} \label{secav}
In this section, we show that the pair $(X,L)$ is unobstructed for any abelian
variety $X$ along with a line bundle $L$ on it. 

Let $X$ be an $n$-dimensional abelian variety over $\C$, the field of complex
numbers. Recall that we have a short exact sequence:
$$
0 \to H^0(X,\Omega^1_X) \to H^1_{dR}(X/\C) \to H^1(X,\mathcal{O}_X) \to 0.
$$

In the following, let us denote $V \coloneqq H^1_{dR}(X/\C)$ and $F \coloneqq
H^0(X,\Omega^1_X)$. Because $V$ has a real structure coming from the
identification $H^1_{dR}(X/\C) \cong H^1_{sing}(X,\R) \otimes_{\R} \C$, we have
a complex conjugation $V \xrightarrow{\overline{(\cdot)}} V$ on $V$. Hodge theory
implies that we have $F \oplus \overline{F} = V$.

From $X$ we get a point $P$ of the Grassmannian $G \coloneqq \Gr(n,V)$
corresponding to $F \subset V$. 
Inside $G$, the linear algebraic condition $F \oplus \overline{F} = V$ coming from Hodge theory defines a complex analytic open 
$\widetilde{G} \coloneqq \{[W \subset V] \in G \mid W \oplus \overline{W} = V\}$ 
known as the \emph{period domain}.

Let $L$ be a line bundle on $X$, let us denote its first Chern class by $c \in
\wedge^2 V$. Let us denote the incidence subvariety 
$\{[W \subset V] \in G \mid \pi_W(c) = 0 \}$ by
$G^{c}$, where $\pi_W \colon \wedge^2 V \to \wedge^2 (V/W)$ denotes the induced
projection. Because $c$ is algebraic, its image in $\wedge^2 (V/F)$ is zero. 
Hence the point $P$ lies in $G^{c}$. Lastly, let us denote the intersection 
$\widetilde{G} \cap G^c \eqqcolon \widetilde{G^c}$.

\begin{thm}
\label{linear algebra smooth}
The complex analytic open $\widetilde{G^c} \subset G^c$ is smooth.
\end{thm}

\begin{warning}
In general $G^c$ is not smooth. 
The complex analytic open $\widetilde{G^c}$, however, is smooth.
\end{warning}

Before proving this Theorem, 
we need some preliminary discussion concerning the related linear algebra.

\begin{defi}
The element $c \in \wedge^2 V \subset V^{\otimes 2}$ induces a linear map 
$V^* \to V$. 
We define the \emph{rank} of $c$ to be the dimension of the image of this linear map,
which we denote by $r(c)$.
\end{defi}

The condition of an element $[W \subset V]$ in the Grassmannian lies in $G^c$
now has a neat linear algebra characterization. 
Indeed, dualizing the short exact sequence
$$
0 \to W \to V \to V/W \to 0,
$$
we get
$$
0 \to W^{\perp}=(V/W)^* \to V^* \to W^* \to 0.
$$
And we have the following:

\begin{claim}
\label{Gc claim}
An element $[W \subset V]$ in the Grassmannian lies in $G^c$
if and only if the composition map $W^{\perp} \to V^* \xrightarrow{c} V \to V/W$
is zero.
\end{claim}

This is basic linear algebra and we omit the proof here. 
Let $Q \in G^c$ be the point corresponding to $[W \subset V]$.
By Claim~\ref{Gc claim} above, the linear map $c \colon V^* \to V$
can be extended to the following diagram:
$$
\xymatrix{
W^{\perp} \ar@{-->}[r]^{c'} \ar[d] & W \ar[d] \\
V^* \ar[r]^{c} \ar[d] & V \ar[d] \\
W^* \ar@{-->}[r]^{c''} & V/W. \\
 \ar @{} [ur] |{\text{diagram } (*)}
}
$$

\begin{remark}
\label{linear algebra remark}
It is worth noting that the condition of $c \in \wedge^2 V$ implies that
$c'$ and $c''$ are negative transpose to each other.
\end{remark}

Finally, we want to understand the condition of an element 
$[W \subset V]$ in the Grassmannian lies in $\widetilde{G^c}$.
Suppose $Q \in \widetilde{G}$ corresponds to $[W \subset V]$.
Now if we fix a basis $\{w_1, \ldots, w_n\}$ of $W$,
we get an induced basis 
$\{w_1, \ldots, w_n, \overline{w_1}, \ldots, \overline{w_n}\}$ of $V$
and the induced dual basis
$\{w_1^*, \ldots, w_n^*, \overline{w_1}^*, \ldots, \overline{w_n}^*\}$ of $V^*$.
Under this set of basis, the linear map $c$ can be represented by
$$
c = 
\begin{pmatrix}
\alpha & \beta \\
\gamma & \delta
\end{pmatrix}.
$$
Recall that $V$ has a real structure, 
therefore it makes sense to say that $c \in \wedge^2 V$ is a \emph{real} element.
The following is a key lemma. 

\begin{lemm}
\label{Gc lemma}
If the point $Q$ corresponding to $[W \subset V]$ lies in $\widetilde{G^c}$,
then $\alpha = \delta = 0$. 
Moreover, the $c'$ (resp.~$c''$) in diagram~(*) is represented by 
$\beta$ (resp.~$\gamma$).
In particular, we have $r(c) = 2 \rank (\beta)$.
\end{lemm}

\begin{proof}
Since $Q$ lies in $\widetilde{G}$, the Hodge decomposition 
induces a natural decomposition of 
$\wedge^2 V = \bigoplus_{i + j = 2; i,j \geq 0} (\wedge^2 V)^{i,j}$. 
The condition of $Q \in \widetilde{G^c}$ implies that $c$ has no $(0,2)$-component.
Since $c$ is a real class, it also has no $(2,0)$-component.
This implies the first statement.
The second statement is immediate.

Now we see that 
$$
c = 
\begin{pmatrix}
0 & c' = \beta \\
c'' = \gamma & 0
\end{pmatrix}.
$$

Lastly, by Remark~\ref{linear algebra remark}, 
we see that the rank of $\beta$ and $\gamma$ are the same. 
Hence the last statement follows.
\end{proof}

In view of Claim~\ref{Gc claim} and Lemma~\ref{Gc lemma}, 
we see that the rank $r(c)$ must be even and the rank of the associated linear map
$c'$ is a constant (half of $r(c)$). 
Now we are ready to give the following:

\begin{proof}[proof of Theorem~\ref{linear algebra smooth}]
Since $\widetilde{G^c} \subset G^c$ is a complex analytic open,
it suffices to show that for any $Q \in \widetilde{G^c}$,
the dimension of the tangent space $\dim T_{G^c,Q}$ is independent of $Q$.
In the following, 
we are going to prove that $\dim T_{G^c,Q}$ solely depends on $r(c)$.

Let $Q \in \widetilde{G^c}$ correspond to $W \subset V$,
and let us adopt the notation in diagram (*).
It is well-known that $T_{G,Q} = \Hom(W, V/W)$.
After chasing through the identification, with the aid of Claim~\ref{Gc claim},
we see that 
$$
T_{G^c,Q} = \{\phi \colon W \to V/W \mid \phi \circ c' = 0\} \\
= \Hom(W/c'(W^{\perp}), V/W) \subset \Hom(W, V/W).
$$
The dimension of this space is 
$(n- \rank(c')) \cdot n$, which is equal to $(n - \frac{r(c)}{2}) \cdot n$ by Lemma~\ref{Gc lemma}.
Therefore the dimension of the tanget space solely depends on $r(c)$. 
Hence we see that $\widetilde{G^c}$ is smooth.
\end{proof}

With Theorem~\ref{linear algebra smooth} in hand,
we are ready to study the deformation functor of the abelian variety $X$
along with the line bundle $L$.
For a preliminary discussion about deformation functor of an abelian variety,
we refer the reader to~\cite[Section 2.2]{Oort} and the references listed therein.
For instance, it is worth noting that the deformation functor of an abelian variety
as a variety is the same as that as a group variety, 
see~\cite[Proposition 2.2.6]{Oort}.

The deformation functor of the abelian variety $X$
is pro-represented by $\widehat{G_P}$, the formal completion of $G$ at $P$
and the forgetful map $\Def(X,L) \to \Def(X) = \widehat{G_P}$
factors through $\widehat{G^c_P}$, the \emph{Hodge locus} associated with $L$.
Therefore we get a diagram:
$$
\xymatrix{
\Def(X,L) \ar[r] \ar[rd]^{f} & \Def(X) = \widehat{G_P} \\
& \widehat{G^c_P} \ar@{^{(}->}[u]
}.
$$

\begin{prop}
\label{deformation surjection}
The map $f \colon \Def(X,L) \to \widehat{G^c_P}$ is surjective.
In fact, it has a section.
\end{prop}

\begin{proof}
It suffices to show that $f$ has a section.
By the inclusion $\widehat{G^c_P} \hookrightarrow \widehat{G_P} = \Def(X)$,
we get a formal abelian scheme $\mathcal{X} \to \widehat{G^c_P}$.
The assertion about section now simply means that we can find
a line bundle $\mathcal{L}$ on $\mathcal{X}$ with the property that
$\mathcal{L}_{\mathcal{X}_P} \cong L$. 
Since $G^c$ is smooth at the point $P$, by Theorem~\ref{linear algebra smooth},
we see that $\widehat{G^c_P}$ is of the form $\mathrm{Spf}(A)$ 
as in the Theorem~\ref{thmlifting}.
Therefore Theorem~\ref{thmlifting} guarantees that we may find the required
formal line bundle $\mathcal{L}$ on $\mathcal{X}$.
\end{proof}

Let $\mathcal{X}$ and $\mathcal{L}$ be as in the proof of 
Proposition~\ref{deformation surjection}.
Consider the dual formal abelian scheme $Pic^0_{\mathcal{X}/\widehat{G^c_P}}$,
and let $\mathcal{Y}$ be the formal completion of $Pic_{\mathcal{X}/\widehat{G^c_P}}$
along its identity section, which is formally smooth 
(c.f.~\cite[Proposition 6.7]{GIT}).
The fibre product (in the category of formal schemes) 
$\mathcal{X} \times_{\widehat{G^c_P}} \mathcal{Y}$ defines a formal abelian
scheme $\mathcal{A}$ over $\mathcal{Y}$. 
We summarize the situation in the following diagram:
$$
\xymatrix{
\mathcal{A} \ar[r]^{\pi} \ar[d] & \mathcal{X} \ar[d] \\
\mathcal{Y} \ar[r] & \widehat{G^c_P}.
}
$$

Let $\mathcal{P}$ be the restriction of the Poincar\'{e} line bundle on $\mathcal{A}$.
The pair $(\mathcal{A},\mathcal{P} \otimes \pi^*(\mathcal{L}))$ defines
a pro-object of $\Def(X,L)$ on $\mathcal{Y}$, 
which gives rise to a morphism $\mathcal{Y} \to \Def(X,L)$.

\begin{thm}
\label{abelian variety unobstructed}
The morphism $g \colon \mathcal{Y} \to \Def(X,L)$ constructed above
is an isomorphism.
In particular, since $\mathcal{Y}$ is formally smooth, the pair $(X,L)$ is unobstructed (see Definition \ref{defunobs}).
\end{thm}

\begin{proof}
These all follow easily from the definition of the dual abelian scheme
and Poincar\'{e} line bundle.
\end{proof}

\section{The main theorem} \label{secmainthm}
Let X be a proper algebraic manifold. Recall that
\begin{enumerate}
  \item $X$ is Calabi--Yau if $\dim(X)$ is at least $3$ and $h^0(X,\wedge^p\Omega_X)=0$ for $0<p<\dim(X)$;
  \item $X$ is irreducible holomorphic symplectic if $X$ is simply-connected and $\Ho^{2,0}(X)$ is spanned by the class of a holomorphic symplectic form $\sigma$.
\end{enumerate}

For a irreducible holomorphic symplectic $X$ of dimension $2m$, we have the following facts, see~\cite{CY} for details.
\begin{enumerate}
\item On $\Ho^2(X,\mathbb{R})$, there is a (Beauville--Bogomolov) quadratic form $q_X$ such that
\[q_X(\alpha)=\frac{m}{2}\int_X \alpha^2 (\sigma \bar{\sigma})^{m-1}+(1-m)\left(\int_X \alpha\sigma^{m-1}\overline{\sigma}^m\right)\left(\int_X \alpha \sigma^m \overline{\sigma}^{m-1}\right).\]
  \item $\Ho^0(X,\wedge^{*}\Omega_X)=\mathbb{C}[\sigma]$.
  \item The Beauville--Bogomolov quadratic form $q_X$ is positive definite on $\mathbb{R}[w]\oplus(\Ho^{0,2}\oplus \Ho^{2,0})(X)|_{\mathbb{R}}$, negative definite on the primitive $(1,1)$-part $\Ho^{1,1}(X)_w$ and these two spaces are orthogonal with respect to $q_X$ where $[w]$ is a K\"ahler class.
\end{enumerate}

\begin{lemm} \label{hk2}
Let $Y$ be a proper algebraic manifold with a line bundle $L$. If $Y$ is simply connected and $c_1(L)(\in \Ho^2(Y,\mathbb{Q}))$ is zero, then $L$ is a trivial line bundle.
\end{lemm}
\begin{proof}
Since we have exact sequence
\[\xymatrix{ \Ho^1(Y,\OO_Y)\ar[r] & \Ho^1(Y,\OO_Y^*)\ar[r]^{c_1} &\Ho^2(X,\mathbb{Z})
}\]
and $Y$ is simply connected (in particular $\Ho^1(Y,\OO_Y)=0$), we have that $L$ is a torsion line bundle, i.e., there exists the smallest $N\in\mathbb{N}$ such that \[L^{\otimes N}=\OO_Y.\]Therefore, there is a finite $N$-covering of $X$ 
\[\underline{\Spec}(\OO_Y\oplus L^{-1}\oplus\ldots \oplus L^{-N+1})\rightarrow Y\]
induced by $L$. Since $Y$ is simply connected, $N$ has to be $1$. 
\end{proof}

\begin{lemm} \label{hk1}
Let $Y$ be a proper algebraic manifold. If $\Ho^1(Y,\OO_Y)=0$, then any deformation $(Y_A,L_A)$ of the pair $(Y,\OO_Y)$ is $(Y_A,\OO_{Y_A})$.
\end{lemm}
\begin{proof}
We prove this lemma by induction on the length $l(A)$ of $A$, with the case of $l(A)=1$ being trivial. In general, 
we consider a small extension \[0\rightarrow\mathbb{C}\rightarrow A'\rightarrow A\rightarrow 0\]
and a flat deformation $Y_{A'}$ of $Y$ over $A'$.
Denote by $Y_A$ the pull-back $(Y_{A'})_{A}$ of $Y_{A'}$ to $A$. We have an exact sequence
\[0 \rightarrow \OO_{Y}\rightarrow \OO_{Y_{A'}}^* \rightarrow \OO_{Y_{A}}^* \rightarrow 0\]
which gives rise to an exact sequence
\begin{equation} \label{exa}
\xymatrix{\Ho^1(Y,\OO_Y)\ar[r] &\Ho^1(Y_{A'},\OO_{Y_{A'}}^*) \ar[r]^{p} & \Ho^1(Y_A,\OO^*_{Y_A})}.
\end{equation}
Suppose that $(Y_{A'},L_{A'})$ is a deformation of $(Y,\OO_Y)$. By the induction, we know $(Y_A, (L_{A'})_A)$ is isomorphic to the deformation $(Y_{A},\OO_{Y_{A}})$. It follows that \[p(Y_{A'},L_{A'})=p(Y_{A'},\OO_{Y_{A'}}).\] Since $\Ho^1(Y,\OO_Y)$ is zero, $(Y_{A'},L_{A'})$ is isomorphic to $(Y_{A'},\OO_{Y_{A'}})$ by the exact sequence (\ref{exa}).
\end{proof}

\begin{prop} \label{unobshk}
Let $Y$ be an irreducible holomorphic symplectic manifold of dimension $n=2m$. If $L$ is a line bundle on $Y$, then 
$(Y,L)$ is unobstructed (see Definition \ref{defunobs}).
\end{prop}

\begin{proof}
If $c_1(L)=0 \in \Ho^2(Y,\mathbb{Z})$, then by Lemmas~\ref{hk1} and ~\ref{hk2}, the pair $(Y_A,L_A)$ is isomorphic to $(Y_A,\OO_{A})$ which is obviously unobstructed.

Note that $\Ho^2(Y,\OO_Y)=\mathbb{C}$. If $c_1(L)\neq 0 \in \Ho^2(Y,\mathbb{Z}) $, then it follows from Lemma~\ref{unobs} that we only need to prove that the map \[\cup c_1(L):\Ho^1(Y,T_Y) \rightarrow \Ho^2(Y,\OO_Y)=\mathbb{C}\]
induced by cup product of $c_1(L)$ is nonzero (hence surjective). By the triviality of the canonical bundle of $Y$, we have $T_Y=\wedge^{n-1} \Omega_Y^1$. Moreover, the map $\cup c_1(L)$ fits into the sequence of maps
\[\xymatrixcolsep{3pc}\xymatrix{
\Ho^0(\wedge^{n-2}\Omega_Y)\ar[r]^{\cup [w]} \ar@/^2pc/[rrr]^s & \Ho^1(\wedge^{n-1}\Omega_Y) \ar[r]^{\cup c_1(L)}& \Ho^2(\OO_Y) \ar@{=}[r] & \Ho^2(\wedge ^n \Omega_Y)
}
\]
where $[w]$ is a K\"ahler class in $\Ho^{1,1}(X)$.
To prove that the map $\cup c_1(L)$ is nonzero, it suffices to show that the map $s$ is nonzero.
Since $\Ho^0(\wedge ^{n-2}\Omega_Y)=\mathbb{C}$ is generated by $\sigma^{m-1}$, the map $s$ sends
\begin{center}
    $\sigma^{m-1}$ to $\sigma^{m-1} (c_1(L))\cdot[w] \in \Ho^2(\wedge ^n\Omega_Y).$
\end{center}
Take the property $(3)$ of $q_X$ at the beginning of this section into consideration. If $c_1(L)$ is not zero, then we can choose a K\"ahler class $[w] \in \Ho^{1,1}(Y)$ such that the Beauville--Bogomolov form \[q_Y([w],c_1(L))=\frac{m}{2}\int_Y (c_1(L))\cdot [w] (\sigma \overline{\sigma})^{m-1}\] is non-zero, see~\cite[Corollary 23.11]{CY}. It implies that $s$ is a nonzero map.
\end{proof}
\textbf{Proof of Theorem~\ref{mainthm}.}
\begin{proof}
By the Beauville--Bogomolov decomposition theorem, see~\cite{BEV} and~\cite{BO}, 
there exists a finite \'etale cover $\widetilde{X}\rightarrow X$ such that $\widetilde{X}$ is a product of 
an abelian variety and factors each of which is either an irreducible holomorphic symplectic variety or a Calabi--Yau variety. 
If $\widetilde{X}$ is an abelian variety, then the theorem follows from Proposition~\ref{thm1} and Theorem~\ref{abelian variety unobstructed}. 
Therefore, we can assume that $\widetilde{X}$ does not only contain a torus factor, 
i.e., we have the decomposition $\widetilde{X}=\big(\Pi_{i=1}^{n} Y_i \big) \times A$ 
where $Y_i$ is either a irreducible holomorphic symplectic manifold 
or a Calabi--Yau manifold of dimension at least $3$ and $A$ is an abelian variety. 

For a Calabi--Yau variety $Y$ of dimension at least $3$, the deformations of a pair $(Y,\widetilde{L})$ are unobstructed since $\Ho^2(Y,\OO_Y)=0.$ Moreover, a product of Calabi--Yau manifolds and irreducible holomorphic symplectic manifolds has trivial canonical bundle and has no nonzero global tangent field. Therefore, the decomposition $\widetilde{X}=\big(\Pi_{i=1}^{n} Y_i \big) \times A$ satisfies the conditions of Proposition~\ref{unobsprod}. Our theorem now follows from Proposition~\ref{thm1}, Proposition~\ref{unobsprod}, Theorem~\ref{abelian variety unobstructed} and Proposition~\ref{unobshk}.
\end{proof}

\appendix
\section{Deformations via cotangent complexes} \label{app}
In this appendix, we give an alternative proof of Proposition~\ref{thm00} and the third assertion of Proposition~\ref{thm1} in a slightly general form by cotangent complexes. The cotangent complex is a very powerful machinery to attack deformation problems, one advantage is its functoriality.
\begin{lemm}
Suppose that we have a small extension \[0\rightarrow (t)\rightarrow A' \rightarrow A\rightarrow 0\]
and cartesian diagrams of algebraic schemes
\[\xymatrix{X_0\ar[r]^{g_0}\ar[d]\ar@{}[dr]|-{\Box} & Y_0 \ar[d] \ar@{}[dr]|-{\Box} \ar[r]^{f_0} &\Spec(\mathbb{C})\ar[d]\\
X_A\ar[r]^{g} & Y_A\ar[r]^f & \Spec(A)}
\]where $g$ is finite $\acute{e}$tale and $f$ is flat and of finite type. Then, there is a natural map $g_0^*$ induced by $g_0$
\[g_0^*:\Ext^2(\mathbb{L}_{Y_0},\OO_{Y_0}) \rightarrow \Ext^2(\mathbb{L}_{X_0},\OO_{X_0})\]
mapping $Obs(Y_A)$ to $Obs(X_A)$, where $\mathbb{L}_{Y_0}$ (resp.~$\mathbb{L}_{X_0}$) is the cotangent complex of $Y_0$ (resp.~$X_0$) over $\mathbb{C}$.
\end{lemm}

\begin{proof}
Since $f$ and $g$ are flat, we have that
\begin{center}
    $(\mathbb{L}_{X_A/A})|_{X_0}=\mathbb{L}_{X_0}$ which implies $\Ext^1(\mathbb{L}_{X_A/A},\OO_{X_0})=\Ext^1(\mathbb{L}_{X_0},\OO_{X_0})$.
\end{center}
Let us recall how to construct obstructions $Obs(X_A)$ and $Obs(Y_A)$, see~\cite{I} for the details. Note that we have a distinguished triangle $\Delta_{Y_A}$ in $D_{coh}^-(Y_A)$ as follows.
\begin{equation} \label{eq1}
\xymatrix{f^*\mathbb{L}_{A/A'}\ar[rr] && \mathbb{L}_{Y_A/A'}\ar[dl]\\
&\mathbb{L}_{Y_A/A}\ar[ul]^{+1}}
\end{equation}
Similarly, we also have a distinguished triangle $\Delta_{X_A}$ in $D_{coh}^-(X_A)$. We apply the functors $\Ext^i(-,f^*(t))=\Ext^i(-,\OO_{Y_0})$ to the triangle $\Delta_Y$. We get a long exact sequence
\[\xymatrix{\ldots\ar[r]  &\Ext^1(f^*\mathbb{L}_{A/A'},\OO_{Y_0})\ar@{=}[d]\ar[r] &\Ext^2(\mathbb{L}_{Y_A/A},\OO_{Y_0})\ar@{=}[d]\ar[r] &\ldots\\
&\mathrm{Hom}(\OO_{Y_0},\OO_{Y_0})\ar[r]^h & \Ext^2(\mathbb{L}_{Y_0},\OO_{Y_0})}\]
where the first vertical identification follows from
\begin{equation}\label{eq2}
_{\tau\geq -1}f^*\mathbb{L}_{A/A'}=f^*(t)[1]=\OO_{Y_0}[1].
\end{equation}
 The obstruction is given by \[\Id_{\OO_{Y_0}}\mapsto Obs({Y_A})=h(\Id_{\OO_{Y_0}}).\]

Since $g$ is \'etale, we have $\mathbb{L}_{X_A/{Y_A}}=0$. It implies that $g^*\Delta_{Y_A}=\Delta_{X_A}$.
Therefore, we have a natural commutative diagram of distinguished triangles
\[\xymatrix{\mathbb{R}\underline{\mathrm{Hom}}_{Y_A}(\Delta_{Y_A},f^*(t))\ar@{=}[d]\ar[r] &\mathbb{R}\underline{\mathrm{Hom}}_{Y_A}(\Delta_{Y_A},g_*g^*f^*(t))\ar@{=}[d]\\
\mathbb{R}\underline{\mathrm{Hom}}_{Y_A}(\Delta_{Y_A},\OO_{Y_0})\ar[r] \ar[dr]^{\delta} & \mathbb{R}g_*\mathbb{R}\underline{\mathrm{Hom}}_{X_A}(g^*\Delta_{Y_A},g^*f^*(t))\ar@{=}[d] \\
&\mathbb{R}g_*\mathbb{R}\underline{\mathrm{Hom}}_{X_A}(\Delta_{X_A},\OO_{X_{0}})}\]
where the second vertical identification follows from~\cite[Proposition 5.10]{H2} and the row map is induced by the adjoint map $f^*(t)\rightarrow g_*g^*f^*(t)$. Note that
\begin{center}
    $\mathbb{H}^i(Y_A,\mathbb{R}\underline{\mathrm{Hom}}(\mathcal{F},\mathcal{G}))=\Ext^i(\F,\mathcal{G})$ \qquad and \qquad $\mathbb{H}^i(Y_A,\mathbb{R}g_*\F)=\mathbb{H}^i(X_A,\F)$.
\end{center}
Applying the hypercohomology functor $\mathbb{H}^*(Y,-)$ to the morphism $\delta$, we get the following commutative diagram.
\[\xymatrix{\mathrm{Hom}(\OO_{Y_0},\OO_{Y_0})\ar[d] \ar[r]^h & \Ext^2(\mathbb{L}_{Y_0},\OO_{Y_0})\ar[d]^{g^*} & 1 \ar@{|->}[r]^h\ar@{|->}[d] & Obs(Y_A)\ar@{|->}[d]^{g^*}\\
\mathrm{Hom}(\OO_{X_0},\OO_{X_0})\ar[r]^h & \Ext^2(\mathbb{L}_{X_0},\OO_{X_0}) & 1 \ar@{|->}[r]^h &Obs(X_A)}
\]
\end{proof}

For a pair $(Y_A,L_A)$, we can show Proposition~\ref{thm1} (3) similarly once we notice that the obstructions are constructed in a similar way. Let us indicate it briefly. Suppose that $$0\rightarrow\mathbb{C}\rightarrow A'\rightarrow A\rightarrow 0$$ is a small extension. For a pair $(Y,V)$ where $V$ is a vector bundle on $Y$ and a deformation pair $(Y_A,V_A)$, the obstruction element
\begin{center}
$Obs(Y_A,V_A)\in \Ext^2(\At(V),V)$
\end{center}
is constructed in the following way
\[\xymatrix{ V_A\otimes \mathbb{L}_{Y_A/A'} \ar[d] \ar[r] & \At'(V_A)\ar[d]\ar[r] &V_A \ar@{=}[d] \\
V_A\otimes \mathbb{L}_{Y_A/A} \ar[d] \ar[r] & \At(V_A) \ar@{-->}[dl] \ar[r] &V_A\\
V_A\otimes f^*\mathbb{L}_{A/A'}[1]}\]
where the first vertical arrows are given by tensoring (\ref{eq1}) with $V_A$, and $\At'(V_A)$ and $\At(V_A)$ are Atiyah extensions, see~\cite{I}. 
By (\ref{eq2}), we have an induced morphism $$\At(V_A)\rightarrow V_A\otimes f^*L_{A/A'}[1]$$ which is an element in 
$$\Ext^1(\At(V_A),V_A\otimes f^*\mathbb{L}_{A/A'})=\Ext^2(\At(V_A),V_A\otimes \OO_Y)=\Ext^2(\At(V),V).$$
This map gives rise to the obstruction element $Obs(Y_A,V_A)\in \Ext^2(\At(V),V)$.\\




\begin{thebibliography}{GHJ03}


\bibitem[Bea83]{BEV}
Arnaud Beauville.
\newblock Vari\'et\'es {K}\"ahleriennes dont la premi\`ere classe de {C}hern
  est nulle.
\newblock {\em J. Differential Geom.}, 18(4):755--782 (1984), 1983.

\bibitem[BEK14]{defcycle}
Spencer Bloch, H\'{e}l\`ene Esnault, and Moritz Kerz.
\newblock Deformation of algebraic cycle classes in characteristic zero.
\newblock {\em Algebr. Geom.}, 1(3):290--310, 2014.

\bibitem[Blo72]{Bloch}
{\sc Bloch, S.}
\newblock Semi-regularity and de{R}ham cohomology.
\newblock {\em Invent. Math. 17\/} (1972), 51--66.

\bibitem[Bog74]{BO}
Fedor Bogomolov.
\newblock The decomposition of {K}\"ahler manifolds with a trivial canonical
  class.
\newblock {\em Mat. Sb. (N.S.)}, 93(135):573--575, 630, 1974.

\bibitem[Del68]{Del}
{\sc Deligne, P.}
\newblock Th\'eor\`eme de {L}efschetz et crit\`eres de d\'eg\'en\'erescence de
  suites spectrales.
\newblock {\em Inst. Hautes \'Etudes Sci. Publ. Math.}, 35 (1968), 259--278.


\bibitem[FK88]{ET}
Eberhard Freitag and Reinhardt Kiehl.
\newblock {\em \'{E}tale cohomology and the {W}eil conjecture}, volume~13 of
  {\em Ergebnisse der Mathematik und ihrer Grenzgebiete (3) [Results in
  Mathematics and Related Areas (3)]}.
\newblock Springer-Verlag, Berlin, 1988.
\newblock Translated from the German by Betty S. Waterhouse and William C.
  Waterhouse, With an historical introduction by J. A. Dieudonn{\'e}.

\bibitem[GH94]{GH1}
Phillip Griffiths and Joseph Harris.
\newblock {\em Principles of algebraic geometry}.
\newblock Wiley Classics Library. John Wiley \& Sons Inc., New York, 1994.
\newblock Reprint of the 1978 original.

\bibitem[GHJ03]{CY}
Mark Gross, Daniel Huybrechts and Dominic Joyce.
\newblock {\em Calabi-{Y}au manifolds and related geometries}.
\newblock Universitext. Springer-Verlag, Berlin, 2003.
\newblock Lectures from the Summer School held in Nordfjordeid, June 2001.

\bibitem[GIT]{GIT}
D.~Mumford, J.~Fogarty, and F.~Kirwan.
\newblock {\em Geometric invariant theory}, volume~34 of {\em Ergebnisse der
  Mathematik und ihrer Grenzgebiete (2) [Results in Mathematics and Related
  Areas (2)]}.
\newblock Springer-Verlag, Berlin, third edition, 1994.

\bibitem[Har66]{H2}
Robin Hartshorne.
\newblock {\em Residues and duality}.
\newblock Lecture notes of a seminar on the work of A. Grothendieck, given at
  Harvard 1963/64. With an appendix by P. Deligne. Lecture Notes in
  Mathematics, No. 20. Springer-Verlag, Berlin, 1966.

\bibitem[Har77]{HA}
Robin Hartshorne.
\newblock {\em Algebraic geometry}.
\newblock Springer-Verlag, New York, 1977.
\newblock Graduate Texts in Mathematics, No. 52.

\bibitem[Ill71]{I}
Luc Illusie.
\newblock {\em Complexe cotangent et d\'eformations. {I}}.
\newblock Lecture Notes in Mathematics, Vol. 239. Springer-Verlag, Berlin,
  1971.

\bibitem[Kaw92]{KAW}
Yujiro Kawamata.
\newblock Unobstructed deformations. {A} remark on a paper of {Z}. {R}an:
  ``{D}eformations of manifolds with torsion or negative canonical bundle''
  [{J}. {A}lgebraic {G}eom.\ {\bf 1} (1992), no.\ 2, 279--291; {MR}1144440
  (93e:14015)].
\newblock {\em J. Algebraic Geom.}, 1(2):183--190, 1992.

\bibitem[Kon]{Kon}
Maxim Kontsevich.
\newblock Generalized Tian--Todorov theorems.
\newblock {\em Preprint}.

\bibitem[Oor71]{Oort}
Frans Oort.
\newblock Finite group schemes, local moduli for abelian varieties, and lifting
  problems.
\newblock {\em Compositio Math.}, 23:265--296, 1971.

\bibitem[Ran92]{RAN}
Ziv Ran.
\newblock Deformations of manifolds with torsion or negative canonical bundle.
\newblock {\em J. Algebraic Geom.}, 1(2):279--291, 1992.

\bibitem[Sch68]{AR}
Michael Schlessinger.
\newblock Functors of {A}rtin rings.
\newblock {\em Trans. Amer. Math. Soc.}, 130:208--222, 1968.

\bibitem[Ser06]{ES}
Edoardo Sernesi.
\newblock {\em Deformations of algebraic schemes}, volume 334 of {\em
  Grundlehren der Mathematischen Wissenschaften [Fundamental Principles of
  Mathematical Sciences]}.
\newblock Springer-Verlag, Berlin, 2006.

\bibitem[Tia87]{TG}
Gang Tian.
\newblock Smoothness of the universal deformation space of compact
  {C}alabi-{Y}au manifolds and its {P}etersson-{W}eil metric.
\newblock In {\em Mathematical aspects of string theory ({S}an {D}iego,
  {C}alif., 1986)}, volume~1 of {\em Adv. Ser. Math. Phys.}, pages 629--646.
  World Sci. Publishing, Singapore, 1987.

\bibitem[Tod89]{TO}
Andrey~N. Todorov.
\newblock The {W}eil-{P}etersson geometry of the moduli space of {${\rm
  SU}(n\geq 3)$} ({C}alabi-{Y}au) manifolds. {I}.
\newblock {\em Comm. Math. Phys.}, 126(2):325--346, 1989.

\bibitem[TV]{V}
Mattia Talpo and Angelo Vistoli.
\newblock Deformation theory from the point of view of fibered categories.
\newblock {\em Preprint}.

\end{thebibliography}
\end{document}